\documentclass{article}
\usepackage[utf8]{inputenc}
\usepackage{ gensymb }
\usepackage{amsmath}
\usepackage{ amssymb }
\usepackage{amsthm}
\usepackage{enumitem}
\usepackage{hyperref}
\newtheorem{G1}{Proposition}
\newtheorem{G2}[G1]{Proposition}
\newtheorem{G3}[G1]{Theorem}
\newtheorem{G31}[G1]{Theorem}
\newtheorem{IMO}[G1]{Corollary}
\newtheorem{IMO1968}[G1]{Corollary}
\newtheorem{G2C1}[G1]{Corollary}
\newtheorem{G2C2}[G1]{Corollary}
\newtheorem{G2C3}[G1]{Corollary}
\newtheorem{Lemma 1}[G1]{Lemma}
\newtheorem{Lemma 2}[G1]{Lemma}
\newtheorem{Lemma 3}[G1]{Lemma}
\title{ON SOME GENERALIZATIONS TO FLOOR FUNCTION IDENTITIES OF RAMANUJAN}
\author{Andrzej Kukla \and Sai Teja Somu}
\newcommand{\F}[1]{\left\lfloor#1\right\rfloor}
\date{\today}
\begin{document}
\maketitle
\begin{abstract}
    We give some generalizations to three identities of Srinivasa Ramanujan involving greatest integer function.
\end{abstract}

\section{Introduction}
Let $\lfloor x \rfloor$ denote greatest integer less than or equal to $x$, Ramanujan proposed three interesting identities involving greatest integer function as a problem to the \textit{Journal of the Indian Mathematical Society}. The problem (See \cite{Ramanujan}) says
\emph{If $n$ is any positive integer, prove that}
\begin{align}
    &\quad\F{\frac n3}+\F{\frac{n+2}6}+\F{\frac{n+4}6}=\F{\frac n2}
    +\F{\frac{n+3}6}\\
    &\quad\F{\frac12+\sqrt {n+\frac12}} =\F{\frac12+\sqrt {n+\frac14}}\\
    &\quad \F{\sqrt{n}+\sqrt{n+1}} =\F{\sqrt {4n+2}}.
\end{align}
Chen in \cite{(ii)} gave short proofs of (2) and (3) and gave some conjectures which were generalizations of (2) and (3).  
We give few generalizations of all the identities (1), (2), (3).

\begin{G1} 
For all integers $a,b$ such that $a\neq 0,\ b>0$ and any real number $x$ we have
\begin{equation}\F{\frac{x}{a}}=\sum_{i=0}^{b-1}\F{\frac{x+ia}{ab}}.\end{equation}
\end{G1}

We will show in section 2 how Proposition 1 is a direct consequence of the Hermite's identity (See \cite{Hermite}), and then show how the identity (1) can be proven from Proposition 1, we also give corollaries to Proposition 1. Next in section 3, we prove the following proposition, and show how identity (2) and other corollaries are derived from Proposition 2.  Let set of natural numbers be set of all positive integers.

\begin{G2}
  Let $f:\mathbb{R}\rightarrow\mathbb{R}$ be a strictly increasing (resp. decreasing) function  whose restriction to integers is in integers, such that the equation
$$f(x)\equiv y \text{ (mod }n),\quad 1<n\in\mathbb{N}$$
has no solutions in integers for $y\in\{a,a+1,\cdots,b-1,b\},\  a\leq b,\ a,b\in\mathbb{Z}$. Then 
$$\F{f^{-1}(kn+a-1)}=\F{f^{-1}(kn+a)}=\cdots=\F{ f^{-1}(kn+b-1)}=\F{ f^{-1}(kn+b)}$$
$$\left(\text{resp. }\F{f^{-1}(kn+a)}=\F{ f^{-1}(kn+a+1)}=\cdots=\F{ f^{-1}(kn+b)}=\F{ f^{-1}(kn+b+1)}\right)$$
where $k$ is an integer such that $f^{-1}$ exists for all $[kn+a-1,kn+b]\cap\mathbb{Z}$ ($\text{resp. } [kn+a,kn+b+1]\cap\mathbb{Z}$).
\end{G2}

In Section 4, we consider two identities of the form (3), and prove theorems 3 and 4.

\begin{G3}
 For all natural numbers $n,k>1$, except for finitely many exceptions of the form $n=\F{\left(\frac{3}{2}\right)^k}$ we have 
\[\F{\sqrt[k]{n}+\sqrt[k]{n+1}}=\F{2\sqrt[k]{n+\frac 12}}.\]
\end{G3}
For proving that there are only finitely many exceptions we used a result of Mahler\cite{Mahler}.

\begin{G31}
For any natural numbers $l,k$ and $x_1,x_2,\cdots,x_l$ such that there exists a prime such that $p\ |\ l,\ p^k\nmid l$, $p\nmid (x_1+x_2+\cdots+x_l)$ then if $n\geq\frac{l^{k-1}(x_1^2+\cdots+x_l^2)}{2}$ we have
$$\F{\sqrt[k]{n+x_1}+\sqrt[k]{n+x_2}+\cdots+\sqrt[k]{n+x_l}}=\F{l\sqrt[k]{n+\frac{x_1+x_2+\cdots+x_l}{l}}}.$$
\end{G31}

\section{Generalization of (1)}
We begin by proving Proposition 1,
$$\F{\frac{x}{a}}=\sum_{i=0}^{b-1}\F{\frac{x+ia}{ab}}.$$
\begin{proof}
From Hermite's identity (See \cite{Hermite}), we have for any real number $t$, natural number $b$,
\[  \sum_{i=0}^{b-1}\F{t+\frac{i}{b}}=\F{bt},\] substitute $t=\frac{x}{ab}$ to get the desired identity,
$$\F{\frac{x}{a}}=\sum_{i=0}^{b-1}\F{\frac{x+ia}{ab}}.$$
\end{proof}
Note that when $a=1$, it is the standard floor function identity $\F{x} = \sum_{i=0}^{b-1}\F{\frac{x+i}{b}}.$

Now let us see a proof of Ramanujan's identity (1) using Proposition 1.
\subsection{Proof of (1) using Proposition 1}
\begin{proof}
Set $a=2,\ b=3$ and $a=3,\ b=2$, $x=n$ in Proposition 1 to get:
\begin{align*}
&\F{\frac n2}=\F{\frac{n}{6}}+\F{\frac{n+2}{6}}+\F{\frac{n+4}{6}},\\
&\F{\frac{n}{6}}+\F{\frac{n+3}{6}} =\F{\frac{n}{3}} .
\end{align*}
Adding both the equations and canceling out $\lfloor\frac{n}{6}\rfloor$ on both sides we get (1).
\end{proof}
\subsection{Corollaries of Proposition 1}
 We have the following corollary of Proposition 1.
    
    \begin{IMO}
     For any $x\in\mathbb{R}$ and natural number $n$ the following identity is true
    \[\sum_{j=0}^{\infty}\sum_{i=1}^{n-1}\F{\frac{x}{n^{j+1}}+\frac{i}{n}}=\left\{\begin{array}{ll}
        \F{x} &\text{for }x\geq 0  \\
        \F{x}+1 &\text{for }x< 0.
    \end{array}\right.\]
    \end{IMO}
    \begin{proof}
     We start by noticing that by Proposition 1 we can rewrite the inside sum as
    $$\sum_{i=1}^{n-1}\F{\frac{x}{n^{j+1}}+\frac{i}{n}}=\sum_{i=0}^{n-1}\F{\frac{x+in^j}{n^{j+1}}}-\F{\frac{x}{n^{j+1}}}=\F{\frac{x}{n^j}}-\F{\frac{x}{n^{j+1}}}.$$
    This implies that
    $$\sum_{j=0}^{\infty}\sum_{i=1}^{n-1}\F{\frac{x}{n^{j+1}}+\frac{i}{n}}=\lim_{k\rightarrow \infty}\sum_{j=0}^{k}\left(\F{\frac{x}{n^j}}-\F{\frac{x}{n^{j+1}}}\right)=\lim_{k\rightarrow \infty}\bigg(\F{x }-\F{\frac{x}{n^{k+1}}}\bigg).$$
    As $\lim_{k\rightarrow \infty}\lfloor\frac{x}{n^{k+1}}\rfloor=0$, if $x\geq 0$ and $\lim_{k\rightarrow \infty}\lfloor\frac{x}{n^{k+1}}\rfloor=-1$ otherwise, we have \[\sum_{j=0}^{\infty}\sum_{i=1}^{n-1}\F{\frac{x}{n^{j+1}}+\frac{i}{n}}=\left\{\begin{array}{ll}
        \F{x} &\text{for }x\geq 0  \\
        \F{x}+1 &\text{for }x< 0.
    \end{array}\right.\]
    \end{proof}
    
     Corollary 5 is a generalization of question 6 from the 1968 International Math Olympiad, which can be proved by setting $n=2$, and $x=m$, when $m$ is natural number
        
        \begin{IMO1968}
        For every natural number $m$, we have
        $$\sum_{k=0}^{\infty}\F{\frac{m+2^k}{2^{k+1}}}=\F{\frac{m+1}{2}}+\F{\frac{m+2}{4}}+\F{\frac{m+4}{8}}+\cdots =m.$$
        \end{IMO1968}

\section{Generalization of (2)}
Now we will give a proof of Proposition 2.

\begin{proof}
Let's assume that the function $f$ is strictly increasing (proof of decreasing case is analogous).  As the equation $$f(x)\equiv y \text{ (mod }n),\quad 1<n\in\mathbb{N}$$
has no solutions in integers for $y\in\{a,a+1,\cdots,b-1,b\},\  a\leq b,\ a,b\in\mathbb{Z}$ implies that, let $c=\lfloor f^{-1}(kn+a-1)\rfloor$, $$ f(c)\leq kn+a-1<kn+a<\cdots<kn+b<f(c+1)$$
for an integer $k$ such that $f$ has an inverse for all elements in $[kn+a-1,kn+b]\cap \mathbb{Z}$. Applying $f^{-1}$ on both sides we obtain
$$c\leq f^{-1}(kn+a-1)<f^{-1}(kn+a)<\cdots<f^{-1}(kn+a)<c+1$$
which implies 
$$\F{f^{-1}(kn+a-1)}=\F{f^{-1}(kn+a)}=\cdots=\F{ f^{-1}(kn+b-1)}=\F{ f^{-1}(kn+b)}.$$ 
\end{proof}
Let us prove Ramanujan's identity (2) using Proposition 2.
\subsection{Proof of (2) using Proposition 2}
\begin{proof}
Let us take a function $f(x):=(2x-1)^2$ defined on the interval $[1,+\infty)$. This function is increasing and it is invertible in its whole range. As
$$f(x)=(2x-1)^2\equiv 2\text{ (mod }4)$$
has no solutions in integers, from Proposition 2

$$\F{f^{-1}(4n+1)}=\F{f^{-1}(4n+2)}.$$
The inverse of $f$ is $f^{-1}(x):=\frac{1+\sqrt{x}}{2}$, therefore
$\F{\frac{1+\sqrt{4n+1}}{2}}=\F{\frac{1+\sqrt{4n+2}}{2}}$ or $$\F{\frac{1}{2}+\sqrt{n+\frac{1}{2}}}=\F{\frac{1}{2}+\sqrt{n+\frac{1}{4}}}.$$
\end{proof}
Let us look at some corollaries of Proposition 2. 
\subsection{Corollaries of Proposition 2}
\begin{G2C1}
Let $a,b$ be any two natural numbers and $n\geq \frac{2b}{a}$ be a natural number then 
\[\F{\frac{b+\sqrt{4n+1}}{a}}=\F{\frac{b+\sqrt{4n+2}}{a}}=\F{\frac{b+\sqrt{4n+3}}{a}}.\]
\end{G2C1}
\begin{proof}
Let $f(x):=(ax-b)^2$ defined on $[\frac{2b}{a},\infty)$ $\left( \text{here } f^{-1}(x)=\frac{b+\sqrt{x}}{a}\right)$, then as a square can only take a value equivalent to $0$ or $1$ modulo $4$, from Proposition 2 we have \[\F{\frac{b+\sqrt{4n+1}}{a}}=\F{\frac{b+\sqrt{4n+2}}{a}}=\F{\frac{b+\sqrt{4n+3}}{a}}.\]
\end{proof}
\begin{G2C2}
Let $a,b$ be any two natural numbers and $n$ be any natural number then we have
$$\F{\frac{b+\sqrt[3]{9n+1}}{a}}=\F{\frac{b+\sqrt[3]{9n+2}}{a}}=\cdots=\F{\frac{b+\sqrt[3]{9n+6}}{a}}=\F{\frac{b+\sqrt[3]{9n+7}}{a}}.$$
\end{G2C2}
\begin{proof}
Let $f(x):=(ax-b)^3$ for natural numbers $a,b$, defined on $\mathbb{Z}$,$\left( \text{here } f^{-1}(x)=\frac{b+\sqrt[3]{x}}{a}\right)$, then as a cube can be equivalent only to $0,1,8$ modulo $9$, from Proposition 2 we have \[\F{\frac{b+\sqrt[3]{9n+1}}{a}}=\F{\frac{b+\sqrt[3]{9n+2}}{a}}=\cdots=\F{\frac{b+\sqrt[3]{9n+6}}{a}}=\F{\frac{b+\sqrt[3]{9n+7}}{a}}.\]
\end{proof}
We can derive floor function identities involving logarithm using Proposition 2.
\begin{G2C3}
Let $a\neq 1,c,m\neq 1,b\geq m-1-c$ be natural numbers then for all natural numbers $n$,
$$\F{\frac{\log_a(a^{m}n+a^{m-1})-c}{b}}=\F{\frac{\log_a(a^{m}n+a^{m-1}+1)-c}{b}}=\cdots=\F{\frac{\log_a(a^{m}n+a^{m}-1)-c}{b}}.$$
\end{G2C3}
\begin{proof}
Let $f(x):=a^{bx+c}$ be defined on the interval $[0,+\infty)$ $\left(\text{here }f^{-1}(x)=\frac{\log_a(x)-c}{b}\right)$. As for any natural number $m\neq1$ the congruence
    $$f(x)\equiv y\text{ (mod }a^m)$$
    has no solutions for $y\in\{a^{m-1}+1,\cdots,a^{m}-1\}$ from Proposition 2 we must have $$\F{\frac{\log_a(a^mn+a^{m-1})-c}{b}}=\F{\frac{\log_a(a^{m}n+a^{m-1}+1)-c}{b}}=\cdots=\F{\frac{\log_a(a^{m}n+a^{m}-1)-c}{b}}.$$
\end{proof}
\section{Generalization of (3)}
Let us now prove Theorem 3. 
\subsection{Proof of Theorem 3}
We will prove it using three Lemmas.
\begin{Lemma 1}
For natural numbers $n,k>1$, we have $$\F{\sqrt[k]{2^kn+2^{k-1}-1}} =\F{\sqrt[k]{2^kn+2^{k-1}}}.$$
\end{Lemma 1}
\begin{proof}
 As $\F{\sqrt[k]{2^kn+2^{k-1}}}\geq \F{\sqrt[k]{2^kn+2^{k-1}-1}}$, for $\F{\sqrt[k]{2^kn+2^{k-1}}}\neq \F{\sqrt[k]{2^kn+2^{k-1}-1}}$ there must be exist an integer $m$ such that, $\sqrt[k]{2^kn+2^{k-1}}\geq m >\sqrt[k]{2^kn+2^{k-1}-1}$ or $2^kn+2^{k-1}\geq m^k> 2^kn+2^{k-1}-1$ which is only possible when $2^kn+2^{k-1}=m^k$ but as $2^{k-1}\mid(2^kn+2^{k-1})$ and $2^{k}\nmid (2^kn+2^{k-1})$ we can see that $2^kn+2^{k-1}$ can never be a perfect $k$th power and therefore cannot be equal to $m^k$. Hence  $\F{\sqrt[k]{2^kn+2^{k-1}-1}} =\F{\sqrt[k]{2^kn+2^{k-1}}}$. 
\end{proof}
\begin{Lemma 2}
For natural numbers $n,k>1$, such that $n\geq 2^{k-3}$ we have $$\sqrt[k]{2^kn+2^{k-1}-1}<\sqrt[k]{n}+\sqrt[k]{n+1}.$$
\end{Lemma 2}
\begin{proof}
    From \textit{AM-GM inequality} we have, 
    $\frac{\sqrt[k]{n}+\sqrt[k]{n+1}}{2}>\sqrt[2k]{n^2+n}.$
     As 
     $\sqrt[2k]{n^2+n}\geq\sqrt[k]{n+\frac{2^{k-1}-1}{2^k}}$
     when $n\geq 2^{k-3}$, we have
    $\frac{\sqrt[k]{n}+\sqrt[k]{n+1}}{2}>\sqrt[k]{n+\frac{2^{k-1}-1}{2^k}}$, multiplying both sides by $2$ we get 
    $$\sqrt[k]{2^kn+2^{k-1}-1}<\sqrt[k]{n}+\sqrt[k]{n+1}.$$
\end{proof}
\begin{Lemma 3}
For natural numbers $n,k>1$ we have $$\sqrt[k]{n}+\sqrt[k]{n+1}<\sqrt[k]{2^kn+2^{k-1}}.$$
\end{Lemma 3}
\begin{proof}
From \textit{generalized mean inequality} for $k$ exponent we have 
    \[\frac{\sqrt[k]{n}+\sqrt[k]{n+1}}{2}<\sqrt[k]{\frac{n + n + 1}{2}}=\sqrt[k]{n+\frac{1}{2}}.\]
    It was strict inequality as $\sqrt[k]{n}\neq \sqrt[k]{n+1}$. Multiply both sides by $2$ to get $$\sqrt[k]{n}+\sqrt[k]{n+1}<\sqrt[k]{2^kn+2^{k-1}}.$$
\end{proof}
Let us prove Theorem 3.
\begin{proof}
If $n\geq 2^{k-3}$ then from Lemma 11 and Lemma 12 we have $\sqrt[k]{2^kn+2^{k-1}-1}<\sqrt[k]{n}+\sqrt[k]{n+1}<\sqrt[k]{2^kn+2^{k-1}}$, from Lemma 10 we get $$\F{\sqrt[k]{n}+\sqrt[k]{n+1}} =\F{\sqrt[k]{2^kn+2^{k-1}}} =  \F{2\sqrt[k]{n+\frac{1}{2}}}.$$
If $n<2^{k-3}$, then as  $2<\sqrt[k]{n}+\sqrt[k]{n+1}<4$
and $2<\F{2\sqrt[k]{n+\frac{1}{2}}}<4$ we have $$\F{\sqrt[k]{n}+\sqrt[k]{n+1}},\ \F{ 2\sqrt[k]{n+\frac{1}{2}}}\in\{2,3\}.$$
but from Lemma 12 when $n<2^{k-3}$, the only possible exceptions theorem would occur when 
\begin{equation}\F{\sqrt[k]{n}+\sqrt[k]{n+1}}=2\quad\text{ and }\quad\F{ 2\sqrt[k]{n+\frac{1}{2}}}=3.\end{equation}
If $\quad\F{ 2\sqrt[k]{n+\frac{1}{2}}}=3$, then $2\sqrt[k]{n+\frac{1}{2}}\geq 3$ or $n\geq \left(\frac{3}{2}\right)^k-\frac 12$ and $\F{\sqrt[k]{n}+\sqrt[k]{n+1}}=2$ implies $n\leq \left(\frac{3}{2}\right)^k$.
Therefore (5) can happen only when $n\in\left[\left(\frac{3}{2}\right)^k-\frac 12,\left(\frac{3}{2}\right)^k\right)$ which implies
 $n=\F{\left(\frac 32\right)^k}.$
 
Suppose $n=\F{\left(\frac 32\right)^k}$ is an exception. From (5), $\sqrt[k]{n}+\sqrt[k]{n+1}<3$, using AM-GM inequality we get $2\sqrt[2k]{n(n+1)}<3$ or $n(n+1)<(\frac{3}{2})^{2k}$ or $(2n+1)^2<4(\frac{3}{2})^{2k}+1$. Hence \begin{equation*}
    \left(\frac{3}{2}\right)^k-\frac{1}{2}\leq n<\sqrt{\left(\frac{3}{2}\right)^{2k}+\frac{1}{4}}-\frac{1}{2}<\left(\frac{3}{2}\right)^k+\frac{1}{8\left(\frac{3}{2}\right)^k}-\frac{1}{2},
\end{equation*}
multiply both sides with $2$ and add $1$ on both sides to get
$$2\left(\frac{3}{2}\right)^k\leq 2n+1 <2\left(\frac{3}{2}\right)^k+\frac{1}{4\left(\frac{3}{2}\right)^k}.$$
Hence \begin{equation}|2\left(\frac{3}{2}\right)^k-(2n+1)|<\frac{1}{4\left(\frac{3}{2}\right)^k}.\end{equation} 
 Now we use (5) of Mahler\cite{Mahler}'s paper that for all $u>v\geq 2,\epsilon>0$ and $\vartheta$ any positive algebraic number the following holds for all but finitely $k$. \begin{equation}|\vartheta\left(\frac{u}{v}\right)^k-p^{\text{*}}|>e^{-\epsilon n},\end{equation}
 put  $u=3,\ v=2,\ \vartheta=2,\ \varepsilon=\log\frac 32$ in (7) to see that (6) can be true for finitely many 
. Therefore there can be at most finitely many exceptions $(k,n)$ with $k>1$ for which 
$$\F{\sqrt[k]{n}+\sqrt[k]{n+1}}\neq\F{2\sqrt[k]{n+\frac 12}}.$$
\end{proof}
\subsection{Proof of (3) using Theorem 3}
\begin{proof}
By setting $k=2$ in Theorem 3 we obtain (3) for all $n\geq 2^{k-3}=\frac{1}{2}$, or (3) is true for all natural numbers $n$.
\end{proof}
Let us now prove Theorem 4.
\subsection{Proof of Theorem 4}
We will prove this by proving following inequalities:
\begin{enumerate}[label=\textbf{\arabic*.}]
    \item $\F{\sqrt[k]{n+x_1}+\sqrt[k]{n+x_2}+\cdots+\sqrt[k]{n+x_l}}\leq\F{l\sqrt[k]{n+\frac{x_1+x_2+\cdots+x_l}{l}}}$
    \item $\sqrt[k]{n+x_1}+\sqrt[k]{n+x_2}+\cdots+\sqrt[k]{n+x_l}\geq l\sqrt[lk]{(n+x_1)(n+x_2)\cdots(n+x_l)}$
    \item For $n\geq\frac{l^{k-1}(x_1^2+\cdots+x_l^2)}{2}$ we have $l\sqrt[lk]{(n+x_1)(n+x_2)\cdots(n+x_l)}\geq l\sqrt[k]{n+\frac{x_1+x_2+\cdots+x_l}{l}-\frac{1}{l^k}}$
    \item $\F{l\sqrt[k]{n+\frac{x_1+x_2+\cdots+x_l}{l}}}=\F{l\sqrt[k]{n+\frac{x_1+x_2+\cdots+x_l}{l}-\frac{1}{l^k}}}$
    
\end{enumerate}
We can see that the above inequalities implies Theorem 4, as if
 $n\geq\frac{l^{k-1}(x_1^2+\cdots+x_l^2)}{2}$ then
\begin{align*}
    \F{\sqrt[k]{n+x_1}+\sqrt[k]{n+x_2}+\cdots+\sqrt[k]{n+x_l}} &\stackrel{\text{\textbf{2.}}}{\geq}\F{l\sqrt[k]{(n+x_1)(n+x_2)\cdots(n+x_l)}}\\
    &\stackrel{\text{\textbf{3.}}}{\geq}\F{l\sqrt[k]{n+\frac{x_1+x_2+\cdots+x_l}{l}-\frac{1}{l^k}}}\\
    &\stackrel{\text{\textbf{4.}}}{=}\F{l\sqrt[k]{n+\frac{x_1+x_2+\cdots+x_l}{l}}},
\end{align*}
 therefore, \begin{equation}\tag{\textbf{5.}}
 \F{\sqrt[k]{n+x_1}+\sqrt[k]{n+x_2}+\cdots+\sqrt[k]{n+x_l}}\geq \F{l\sqrt[k]{n+\frac{x_1+x_2+\cdots+x_l}{l}}}   
\end{equation}
Now \textbf{1.} and \textbf{5.} imply the desired result
$$\F{\sqrt[k]{n+x_1}+\sqrt[k]{n+x_2}+\cdots+\sqrt[k]{n+x_l}}=\F{l\sqrt[k]{n+\frac{x_1+x_2+\cdots+x_l}{l}}}.$$
Proofs of \textbf{1.} to \textbf{4.} are given:
\begin{enumerate}[label=\textbf{\arabic*.}]
\item Consider $f(x)=\sqrt[k]{n+x}$. As $f''(x)\leq 0$ from Jensen's inequality we have 
$$\sqrt[k]{n+x_1}+\sqrt[k]{n+x_2}+\cdots+\sqrt[k]{n+x_l}\leq l\sqrt[k]{n+\frac{x_1+x_2+\cdots+x_l}{l}}$$
and therefore 
$$\F{\sqrt[k]{n+x_1}+\sqrt[k]{n+x_2}+\cdots+\sqrt[k]{n+x_l}}\leq\F{l\sqrt[k]{n+\frac{x_1+x_2+\cdots+x_l}{l}}}.$$
\item This one follows from the \textit{AM-GM inequality}.
\item Let $N:=\frac{l^{k-1}(x_1^2+\cdots+x_l^2)}{2}$. Now for $n\geq N$ we have
$$\frac{(x_1^2+\cdots+x_l^2)}{2n^2}\leq \frac{1}{l^{k-1}n}\quad \text{ or }\quad -\frac{(x_1^2+\cdots+x_l^2)}{2n^2}\geq -\frac{1}{l^{k-1}n}.$$
Add $\frac{x_1+\cdots+x_l}{n}$ to both sides to obtain
$$\frac{x_1+\cdots+x_l}{n}-\frac{(x_1^2+\cdots+x_l^2)}{2n^2}\geq \frac{x_1+\cdots+x_l}{n}-\frac{1}{l^{k-1}n}$$
and therefore
$$\left(\frac{x_1}{n}-\frac{x_1^2}{2n^2}\right)+\left(\frac{x_2}{n}-\frac{x_2^2}{2n^2}\right)+\cdots+\left(\frac{x_l}{n}-\frac{x_l^2}{2n^2}\right)\geq l\left(\frac{\frac{x_1+\cdots+x_l}{l}-\frac{1}{l^k}}{n}\right).$$
Now, using elementary inequality of $\log(1+x)$  
$x-\frac{x^2}{2}\leq \log(1+x) \leq x,$ for $0<x\leq 1$ we have 
\begin{align*}
  \log\left(1+\frac{x_1}{n}\right)+\log\left(1+\frac{x_2}{n}\right)+\cdots+\log\left(1+\frac{x_l}{n}\right)&\geq \left(x_1-\frac{x_1^2}{2n^2}\right)+\cdots+\left(x_l-\frac{x_l^2}{2n^2}\right) \\ 
  &\geq l\left(\frac{\frac{x_1+\cdots+x_l}{l}-\frac{1}{l^k}}{n}\right)\\ 
  &\geq l\log\left(1+\frac{\frac{x_1+\cdots+x_l}{l}-\frac{1}{l^k}}{n}\right),
\end{align*}
Add $l\log n$ to both sides of inequality and exponentiate to get
$$(n+x_1)(n+x_2)\cdots(n+x_l)\geq \left(n+\frac{x_1+\cdots+x_l}{l}-\frac{1}{l^k} \right)^l.$$
Now take $\frac{1}{lk}$ powers on both sides and multiply them by $l$ to get the desired result.
\item We claim that  $l\sqrt[k]{n+\frac{x_1+x_2+\cdots+x_l}{l}}$ is never an integer. Note that 
$$l\sqrt[k]{n+\frac{x_1+x_2+\cdots+x_l}{l}}=\sqrt[k]{l^{k-1}(ln+x_1+x_2+\cdots+x_l)}.$$
Now let $\nu_p(a)$ denote the highest power of $p$ that divides $a$, and let $\nu_p(l)=m$. Since $p\ |\ l$ and $\ p^k\nmid l$ we have $1\leq m < k$. As $p\nmid (x_1+x_2+\cdots+x_l)$ and $p\ |\ ln$ we must have $p\nmid (ln+x_1+x_2+\cdots+x_l)$, so 
$$\nu_p(l^{k-1}(ln+x_1+x_2+\cdots+x_l))=(k-1)m.$$
Now $k\nmid (k-1)m$ as $k\ |\ (k-1)m$ implies that $k\ |\ m$, which is not possible as $1\leq m<k$. Hence 
$$k\nmid \nu_p(l^{k-1}(ln+x_1+x_2+\cdots+x_l))$$
so $l^k-1(ln+x_1+x_2+\cdots+x_l)$ cannot be a perfect $k$-th power and consequently $l\sqrt[k]{n+\frac{x_1+x_2+\cdots+x_l}{l}}$ cannot be an integer.

Now let us suppose for the sake of contradiction that \textbf{4.} is not true and therefore there exists an integer $a$ such that
$$ l\sqrt[k]{n+\frac{x_1+x_2+\cdots+x_l}{l}-\frac{1}{l^k}}<a\leq l\sqrt[k]{n+\frac{x_1+x_2+\cdots+x_l}{l}},$$
but we just proved that RHS can't be an integer, so we can say that 
$$ l\sqrt[k]{n+\frac{x_1+x_2+\cdots+x_l}{l}-\frac{1}{l^k}}<a< l\sqrt[k]{n+\frac{x_1+x_2+\cdots+x_l}{l}}.$$
Now raise all expressions to the $k$-th power to get
\begin{equation}\tag{$\star$}
l^kn+l^{k-1}(x_1+x_2+\cdots+x_l)-1<a^k<l^k+l^{k-1}(x_1+x_2+\cdots+x_l)    
\end{equation}
In ($\star$) LHS and RHS are consecutive integers and $a^k$ is also an integer, but since between two consecutive integers there cannot be an integer $(\star)$ is impossible. Hence \textbf{4.} is true.
\end{enumerate}

\section{Acknowledgments}
We would like to thank anonymous user of mathoverflow user142929 who partly conjectured Theorem 3.

\end{document}